\documentclass[a4paper,12pt]{article} 

\usepackage{amsmath,amssymb,amsthm}
\usepackage{latexsym}
\usepackage{graphicx}
\usepackage{epsfig}
\usepackage{float,enumerate}
\usepackage[left=25mm, right=25mm, 
            top=25mm, bottom=25mm]{geometry}
\newtheorem{theo}{Theorem}
\newtheorem{theoetoile}{Theorem}

\newtheorem{lem}[theo]{Lemma}

\newtheorem{coro}[theo]{Corollary}

\newcommand{\bbR}{\mathbb{R}} 
\newcommand{\bbZ}{\mathbb{Z}} 
\newcommand{\set}[1]{\left\{#1\right\}}
\newcommand{\card}{\mathrm{card}}

\newcommand{\eps}{\varepsilon}
\newcommand{\lf}{\lfloor}
\newcommand{\rf}{\rfloor}

\title{Construction of a short path in high dimensional First Passage Percolation}
\author{Olivier Couronn\'e, Nathana\"el Enriquez, Lucas Gerin}

\begin{document}
\maketitle
\begin{abstract}
For First Passage Percolation in $\bbZ^d$ with large $d$, we construct a path connecting the origin to $\set{x_1 =1}$, whose passage time has optimal order $\log d/d$. Besides, an improved lower bound for the "diagonal" speed of the cluster combined with a result by Dhar (1988) shows that the limiting shape in FPP with exponential passage times (and thus that of Eden model) is not the euclidian ball in dimension larger than $35$.
\end{abstract}



\section{Introduction}
\newcommand{\dPerco}{\mathcal{D}}
Let $\set{\tau(x,y),(x,y)\mbox{ edges of }\bbZ^{d}}$ be a family of i.i.d. positive random variables. The quantity $\tau(x,y)$ is called the \emph{passage time} through edge $(x,y)$. For a 
\emph{path} $\mathcal{P}:x_0\to x_1 \to\dots \to x_n$ of neighbouring vertices, we denote by
$\tau(\mathcal{P})$ the passage time along $\mathcal{P}$:
$
\tau(\mathcal{P})=\sum_{i=1}^n \tau(x_{i-1},x_i).
$
The family $\set{\tau(x,y)}$ defines a random distance over $\bbZ^{d}$ as follows:
$$
\dPerco (x,y)=\inf \set{\tau(\mathcal{P})\ ;\ \mathcal{P}\text{ goes from $x$ to $y$. }}.
$$
We also set $B_t:= \set{x;\dPerco(0,x)\leq t}$.

This model is called first passage percolation. We refer to Kesten's St-Flour Lecture Notes \cite{Kes}
for a nice introduction to the subject.
In this work we focus on the case where the common distribution of the passage times is the exponential distribution with parameter one. This case has received a particular attention for at least two reasons:
\begin{itemize}
\item The process $t\mapsto B_t$ is then a Markov process. This is a consequence of the memorylessness
property of the exponential random variable.
\item Consider the random process $B_t$ only at random times at which a new vertex is added. Namely, set $t_0=0$ and $t_{k+1}=\inf \set{t>t_k\ ;\ B_t \neq B_{t-}}$ and look at the sequence $B_{t_0},B_{t_1},\dots$. This discrete process is 
known as the Eden growth process \cite{Eden}, which was introduced as a (very) simplified model for cells spread: to $B_{t_k}$ one adds a new vertex $x$ adjacent to $B_{t_k}$, with a probability which is proportional to the number of edges between $x$ and $B_{t_k}$.
\end{itemize}
We denote the canonical basis of $\bbR^{d}$ by $(e_1,\dots,e_{d})$.
For a subset $A\subset \bbR^{d}$, the random variable $T(A)$ is the first time 
the cluster $B_t$ hits $A$, that is, 
$$
T(A)=\min\set{t\geq 0\ ;\  B_t\cap A\neq \emptyset}.
$$
We will consider $T(A)$ for the particular set
$H_n=\set{x_{1} =n}$.
Subadditivity theory was historically introduced by Hammersley and Welsh 
to show that the sequence $T(n,0,\dots,0)/n$
converges almost surely to a constant $\mu=\mu(d)$, which is called the \emph{time constant}. 
It can be seen as a consequence of the work of Cox and Durrett \cite{Cox} that $\mu$ is also the limit
$$
\mu= \lim_{\mbox{a.s.}} \frac{T(H_n)}{n}.
$$
Kesten (\cite{Kes}, Th.8.2) was the first to prove that, for a large class of distributions for the $\tau$'s, the constant $\mu$ is of order $\log d/d$. His result was later improved by Dhar for exponential passage times.
\begin{theoetoile}[Dhar \cite{Dhar}]
For exponential passage times,
$$
\lim_{d\to \infty} \mu(d)\frac{d}{\log d}  = \frac{1}{2}.
$$
\end{theoetoile}
Some numerical computations are included in \cite{Dhar}, showing that the convergence is quite fast.
Our aim in the present paper is to provide for a constructive proof of the fact that $\mu =\mathcal{O}(\log d/d)$, by
exhibiting a path going from the origin to $H_1$ with small passage time. This path, except its last edge, is included in $H_0$.


\section{A path $\mathcal{P}$ connecting $0$ to $\set{x_1 =1}$}
\subsection*{A tree-like construction of paths}
\newcommand{\dd}{\mathbf{d}}
\newcommand{\Los}[1]{\mathcal{L}^{#1}_a}
\newcommand{\Binom}{\mathcal{B}}

Fix an integer $\ell \geq 2$, and $p_1,..., p_\ell$ a collection of $\ell$ integers larger than 1. 
We want to associate to these integers a path 
${\mathcal P}(\ell, p_1, ..., p_\ell)$, having $\ell +1$ edges and 
connecting the origin to $H_1$, along which passage times are as small as possible.

The first $\ell$ edges
of this path lie in hyperplane $H_0$, and the last one connects $H_0$ to $H_1$.
The path  ${\mathcal P}(\ell, p_1, ..., p_\ell)$ is defined as follows:
\begin{itemize}
\item[Step 1] Among the $2(d-1)$ edges $(0,0\pm e_i)$ (with $2\leq i\leq d$), we consider the ones with the 
$p_1$ smallest passage times. This gives $p_1$ first edges ending at some vertices that we denote by
$x[1],\dots,x[p_1]$.

\item[Step 2] From each point $x_{1}[1],\dots,x_{1}[p_{1}]$, consider the $2(d-1)-2p_1$ edges 
that are not collinear with one of the edges that have been already used in Step 1.
Among these $p_{1}\times(2(d-1)-2p_1)$ distinct edges we choose the $p_{2}$ smallest ones. They end at some 
\emph{distinct} vertices that we denote by $x_{2}[1],\dots,x_{2}[p_{2}]$.

\item[...]
\item[Step $\ell$] From each point $x_{\ell -1}[1],\dots,x_{\ell -1}[p_{\ell -1}]$, consider the $2(d-1)-2p_1-2p_2-\dots -2p_{\ell -1}$ edges 
that are not collinear with one of the edges that have been already used in the previous steps.
Among these $p_{\ell -1}\times(2(d-1)-2p_1-2p_2-\dots -2p_{\ell -1})$ distinct edges we choose the $p_{\ell}$ smallest ones. They end at some 
\emph{distinct} vertices that we denote by $x_{\ell}[1],\dots,x_{\ell}[p_{\ell}]$.

\item[Step  $\ell+1$.] Among the  $p_\ell$ edges $(x_{\ell}[1],x_{\ell}[1]+e_0),\dots,(x_{\ell}[p_{\ell}],x_{\ell}[p_\ell]+e_0)$
we choose the one with the shortest passage time. 
We denote this path by $ x_\ell\to x_{\ell+1}$.
\end{itemize}
Backtracking from $x_{\ell+1}$ to $0$ defines our path: this is the only path
$$
\mathcal{P}(\ell, p_1,..., p_\ell):0\to x_1\to x_2\to x_3\to \dots \to x_{\ell-1}\to x_{\ell}\to x_{\ell+1}
$$
for which, for all $1\leq i\leq \ell -1$, the edge $(x_i,x_{i+1})$ is of the type $(x_i[r],x_{i+1}[s])$ for some 
integers $r,s$. 

\subsection*{The passage time of $\mathcal{P}$}

Now our main result states that, for large $d$, one can find among the paths $
\mathcal{P}(\ell, p_1,..., p_\ell)$ a path whose passage time is of optimal order:
\begin{theo}
$$
\limsup_{d\to\infty} \frac{d}{\log d} \inf_{\ell, p_1,..., p_\ell}\mathbb{E}[\tau(\mathcal{P}(\ell, p_1,..., p_\ell))]\leq \frac{e}{4}.
$$
\end{theo}

\begin{proof}
Let us first introduce a notation. Fix two positive integers $n\geq k\geq0$ and take ${\bf e}_1,{\bf e}_2,\dots, {\bf e}_n$ 
a family of i.i.d. exponential random variables with parameter one. Pick uniformly one of the $k$ smallest, we denote by $f(n,k)$ its expectation. 
The mean passage time of $\mathcal{P}(\ell, p_1,..., p_\ell)$ can be written
\begin{multline}\label{Eq:AOptimiser}
\mathbb{E}[\tau(\mathcal{P}(\ell, p_1,..., p_\ell))]=
f(2(d-1),p_1)+f\left(p_1(2(d-1)-2p_1),p_2\right) +\dots \\
+f\left(p_{\ell -2}(2(d-1)-2\sum_{i=1}^{\ell -1}p_i),p_{\ell}\right)
+f(p_{\ell},1).
\end{multline}  
From the well-known representation of the order statistics of ${\bf e}_1,\dots,{\bf e}_n$ 
$$
\left({\bf e}_{(1)},{\bf e}_{(2)},\dots,{\bf e}_{(n)}\right)\stackrel{\text{(law)}}{=} 
\left(\frac{1}{n} {\bf e}_1, \frac{1}{n} {\bf e}_1 +\frac{1}{n-1} {\bf e}_2 , \dots, \frac{1}{n} {\bf e}_1 +\frac{1}{n-1} {\bf e}_2 +\dots + {\bf e}_n \right).
$$
one readily deduces that 
$$ f(n,k)={1\over k}\sum_{i=1}^k\sum_{j=0}^{i-1}{1\over n-j}\leq {k+1\over 2(n-k)},$$
which implies from (\ref{Eq:AOptimiser}) that 
\begin{multline*}
\mathbb{E}[\tau(\mathcal{P}(\ell, p_1,..., p_\ell))]\leq
\frac{1}{2}\big( \frac{p_1+1}{2(d-1)-p_1} + \frac{p_2+1}{2(d-1)p_1-(2p_1^2 +p_2)} + \dots\\
 + \frac{p_{\ell }+1}{2(d-1)p_{\ell -1}-(2p_{\ell -1}\sum_{i\leq \ell -1}p_i +p_{\ell })}\big)
+ {1\over p_\ell}.
\end{multline*}  
We asymptotically minimize the above right-hand side by introducing a positive integer $A$ and taking $\ell =\lf \log d\rf-A$, and $p_i=\lf e^i\rf$. We obtain  
$$\limsup_{d\to+\infty}{d\over \log d}\inf_{\ell,p_1,...,p_\ell}\mathbb{E}[\tau(\mathcal{P}(\ell, p_1,..., p_\ell))]\leq
\frac{e}{4}\left( {1\over 1-{e\over e-1}e^{-A}}\right)
$$
which gives the desired bound when $A$ grows to infinity.
\end{proof}
If we use the same procedure to build a path from $H_1$ to $H_2$, $H_2$ to $H_3$,..., $H_{n-1}$ to $H_n$,
the family of passage times of these $n$ paths is i.i.d. by construction. It then follows
from the law of large numbers that 
$$
\limsup_{d\to\infty } \mu \frac{d}{\log d}\leq e/4=0.679...
$$

\subsection*{Comments on the result}
\begin{enumerate}
\item We obtain a short and constructive proof of the bound $\mu \leq c^{\text{st}}\log d/d$. We are however not able 
to achieve Dhar's optimal bound with the constant $1/2$.
The latter was obtained with a recursive argument
applied to $B_t$, but cannot provide for an effective path going to $H_n$.
\item Kesten's original proof of the existence of a path whose time constant is less than some constant (11 in his proof) times  ${\log d\over d}$ was also non constructive. However, it indicates  that a path of length $\log d$ achieves this optimal order. This coincides with our choice of $\ell$ in the proof of Theorem 1.  If one restricts the scope to paths of length 3 ($\ell=2$), one already gets an interesting bound i.e. 
$$
\limsup_{d\to\infty} d^{2\over3} \inf_{ p_1,p_2}\mathbb{E}[\tau(\mathcal{P}(2, p_1, p_2))]\leq C
$$
which proves that, in dimension large enough, the horizontal speed is bigger than the diagonal speed, which has been proved to be of order $\sqrt d$ (see next section).
More generally, optimizing $\mathbb{E}[\tau(\mathcal{P}(\ell, p_1,..., p_\ell))]$ for a {\it fixed} $\ell$
leads to a bound  $\mu(d)\leq C/d^{\ell\over\ell+1}$.

\item Our result could be extended to a large class of distributions over the passage times, provided one has a good upper bound for 
$f(n,k)$. This can be done if $\tau$ has a first moment and a nice density near zero, not null at zero (such assumptions on $\tau$ were considered by Kesten).
\end{enumerate}



\section{Discussion on the diagonal speed and the limiting shape}

Richardson \cite{Ric} proved that $B_t$ grows linearly and has a limit shape:
 there exists a nonrandom set $B_0 \subset \bbR^{d+1}$ such that, for all $\eps>0$, 
$$
\mathbb{P}\left((1-\eps)B_0 \subset \frac{B_t}{t} \subset (1+\eps) B_0\right)\stackrel{t\to\infty }\to 1.  
$$
(Richardson didn't exactly deal with first passage percolation but with a class of discrete growth processes, including Eden's growth process.)
The convergence also holds almost surely, the most general result is due to Cox and Durrett \cite{Cox}.

The shape $B_0$ appears to be a ball of a certain norm, which is not explicit. Eden and Richardson observed on simulations that $B_0$ looks circular in dimension two (though, they only performed the simulations up to a few hundreds vertices in $B_t$).
Kesten has shown that, surprisingly enough, this is not the case for FPP with exponential passage times, 
at least when $d>650000$.  
We conclude this paper by short arguments showing that $B_0$ is not the euclidian ball when $d\geq 35$.
Let us denote by $\mu^\star$ the "diagonal" time constant\footnote{We do not use the convention of Kesten \cite{Kes} for the definition of $\mu^\star$, yielding to a different factor of $\sqrt{d}$ between his statement and ours.}:
$$
\mu^\star =\lim_{\mbox{a.s.}} \frac{T(\mathcal{J}_n)}{n},
$$
where $\mathcal{J}_n= \set{x_1+ x_2+\dots +x_{d}=n\sqrt{d} }$. Kesten (\cite{Kes} Th.8.3) observed that $\mu^\star \geq 1/2e\sqrt{d}=0.184.../\sqrt{d}$ which, compared to $\mu =\mathcal{O}(\log d/d)$, gives that $\mu <\mu^\star$ if $d$ is large enough, yielding that $B_0$ is not the Euclidean ball. Carrying his argument a little further, we obtain a slightly improved bound for $\mu^\star$:
\begin{theo}\label{Th:Diago}
For all $d\geq 2$,
\begin{equation}\label{Eq:mustar}
\mu^\star\geq \frac{\sqrt{\alpha_\star^2-1}}{2\sqrt{d}}\approx {0,3313...\over\sqrt d}
\end{equation}
where $\alpha_\star$ is the non null solution of $\coth\alpha=\alpha$.
\end{theo}

\begin{proof}
\newcommand{\Dk}{D_k^{(n)}}
The proof is elementary, it mainly consists in 
bounding the probability that a fixed path going from the origin to $\mathcal{J}_n$ has small passage times.

We denote by $D_k^{(n)}$ the set of self-avoiding paths of length $k$ (in the sense that they do not run twice
through the same edge) starting from zero and hitting $\mathcal{J}_n$ for the first time at time $k$. 
Because of self-avoidingness, passage times are independent along such a path.
Fix a real number $x>0$, since a path from the origin to $\mathcal{J}_n$ has at least $n\sqrt{d}$ edges,
\begin{align}
\mathbb{P}(T(\mathcal{J}_n)\leq n x)
&\leq \mathbb{P}(\mbox{ there exists $k$ and a path $\mathcal{P}$ in $\Dk$ s.t. }\tau(\mathcal{P})\leq nx)\notag\\
&=\sum_{k\geq n\sqrt{d}}\Dk \times \mathbb{P}(\Gamma(k,1) \leq nx)\label{Eq:MajorationDiagonale},
\end{align}
where $\Gamma(k,1)$ is a Gamma$(k,1)$ random variable.
The following estimate is straightforward:
$$
\mathbb{P}(\Gamma(k,1) \leq a)\leq \left(ae/k\right)^k.
$$

\begin{lem}\label{Lem:Dk}
For $k\sim \alpha n\sqrt{d}$ with some constant $\alpha \geq 1$,
$$
\card(\Dk)\sim_{k\to \infty} (2d)^{k} \sqrt{\frac{1}{2\pi n\sqrt{d}}}\left(\frac{\alpha}{(\alpha+1)^{(\alpha+1)/2\alpha}(\alpha-1)^{(\alpha-1)/2\alpha}} \right)^{k}.
$$
\end{lem}
\begin{proof}[Proof of Lemma \ref{Lem:Dk}]
We evaluate the number of (non necessarily self-avoiding) paths of 
length $k$ which start from the origin and hit $\mathcal{J}_n$ for the first time at time $k$.
Such a path
$$
\mathbf{S}:0\to S_1\to \dots\to S_k
$$
is seen as one sample of the standard symmetric random walk in $\bbZ^d$. Its projection
$\mathbf{X}=(0,X_1,\dots,X_k)$ on the axis $x_1=x_2=\dots =x_d$ is a symmetric one-dimensional random walk with increments $\pm 1/\sqrt{d}$. Applying Th.9.1 in \cite{Rev} and the Stirling formula gives
\begin{align*}
(2d)^{-k}\Dk&= \mathbb{P} (\mathbf{S}\mbox{ hits }\mathcal{J}_n \mbox{ for the first time at }k)\\
&=\frac{n\sqrt{d}}{k} \binom{k}{(k+n\sqrt{d})/2}2^{-k}\\
&\sim \sqrt{\frac{1}{2\pi n\sqrt{d}}}\left(\frac{\alpha}{(\alpha+1)^{(\alpha+1)/2\alpha}(\alpha-1)^{(\alpha-1)/2\alpha}} \right)^{k}.
\end{align*}
\end{proof}
Going back to \eqref{Eq:MajorationDiagonale} gives
\begin{align*}
\mathbb{P}(T(\mathcal{J}_n)\leq n x)
&\leq 2\sum_{k\geq n\sqrt{d}} (2d)^k\sqrt{\frac{1}{2\pi n\sqrt{d}}}\left(\frac{\alpha}{(\alpha+1)^{(\alpha+1)/2\alpha}(\alpha-1)^{(\alpha-1)/2\alpha}} \right)^{k} (nxe/k)^k,\\
&\leq 2\sum_{k\geq n\sqrt{d}}  \sqrt{\frac{1}{2\pi n\sqrt{d}}}\left(\frac{2\sqrt{d}xe}{(\alpha+1)^{(\alpha+1)/2\alpha}(\alpha-1)^{(\alpha-1)/2\alpha}} \right)^{k},
\end{align*}
where $\alpha =k/n\sqrt{d}$. This sum decays exponentially provided that
$$
x< \frac{1}{2e\sqrt{d}}\ \sup_{\alpha >1} \set{(\alpha+1)^{(\alpha+1)/2\alpha}(\alpha-1)^{(\alpha-1)/2\alpha}}.
$$
This supremum is attained for the unique non null solution $\alpha_\star$ of $\coth\alpha =\alpha$. It is equal to 
$e\sqrt{\alpha_\star^2-1}$.
Theorem \ref{Th:Diago} then follows from the Borel-Cantelli Lemma.
\end{proof}

\begin{coro}\label{Th:optimalD}
In dimension $d\geq35$, the limiting shape is not the Euclidean ball.
\end{coro}
\begin{proof}
Combining Theorem \ref{Th:Diago} with Dhar's numerical computations we obtain for $d=35$ that $\mu(35)\leq 0.93 \log(2d)/2d< 0.3313 /\sqrt{d}\leq \mu^\star(35)$.
\end{proof}

\subsection*{Aknowledgements}
We are indebted to Olivier Garet who kindly introduced us to the subject, and for very instructive comments on
a preliminary version of this paper. We warmly thank Deepak Dhar for explaining us the
lower bound in \cite{Dhar}.
The second author would like to thank ANR project MEMEMO for financial support.

\vfill

\begin{tabular}{l l l}
Olivier \textsc{Couronn\'e}      & Nathana\"el \textsc{Enriquez}       & Lucas \textsc{Gerin} \\
\texttt{ocouronn@u-paris10.fr} & \texttt{nenriquez@u-paris10.fr} & \texttt{lgerin@u-paris10.fr}
\end{tabular}

\vspace{6mm}

\noindent\textsc{Universit\'e Paris-Ouest}, Laboratoire \textsc{Modal'X}\\
200 avenue de la R\'epublique, 92001 \textsc{Nanterre.}
\end{document}